\documentclass[11pt]{amsart}

\usepackage{amsmath, amscd, amssymb}
\usepackage[frame,cmtip,arrow,matrix,line,graph,curve]{xy}
\usepackage{graphpap, color}
\usepackage[mathscr]{eucal}
\usepackage{cancel}
\usepackage{verbatim}

\numberwithin{equation}{section}

\def\sO{{\mathscr O}}

\newcommand{\CC}{\mathbb{C}}

\newcommand{\LL}{\mathbb{L}}
\newcommand{\PP}{\mathbb{P}}
\newcommand{\QQ}{\mathbb{Q}}

\newcommand{\cal}{\mathcal}

\def\cE{{\cal E}}
\def\cF{{\cal F}}

\def\cK{{\cal K}}

\def\cN{{\cal N}}

\def\fC{\mathfrak{C}}

\def\fE{\mathfrak{E}}

\def\fM{\mathfrak{M}}

\def\loc{\mathrm{loc}}

\let\fE=\cE

\def\and{\quad{\rm and}\quad}
\def\lra{\longrightarrow }
\def\mapright#1{\,\smash{\mathop{\lra}\limits^{#1}}\,}

\let\ti=\tilde
\let\sub=\subset

 \DeclareMathOperator{\Ext}{Ext}

\newtheorem{prop}{Proposition}[section]
\newtheorem{theo}[prop]{Theorem}
\newtheorem{lemm}[prop]{Lemma}

\newtheorem{rema}[prop]{Remark}
\newtheorem{exam}[prop]{Example}

\newtheorem{defi}[prop]{Definition}

\def\beq{\begin{equation}}
\def\eeq{\end{equation}}

\def\lsta{_\ast}

\def\upmo{^{-1}}
\def\sta{^\ast}

\def\dual{^{\vee}}
\def\Ob{{\mathcal Ob}}
\def\virt{^{\mathrm{vir}} }
\def\virtloc{\virt_\loc}

\def\bbL{\mathbb{L} }
\def\DM{Deligne-Mumford }
\def\defeq{{\,:=\,}}
\def\bbA{\mathbb{A} }
\def\zero{\mathrm{zero} }
\def\Po{\PP^1}

\def\redd{{\mathrm{red}}}
\def\TT{\LL\dual}
\def\ep{\epsilon}

\title[Cosection localized virtual cycles]{Torus localization and wall crossing for cosection localized virtual cycles}
\author{Huai-Liang Chang}
\address{Department of Mathematics, Hong Kong University of Science and Technology
Clear Water Bay, Kowloon, Hong Kong}
\email{mahlchang@ust.hk}

\author{Young-Hoon Kiem}
\address{Department of Mathematics and Research Institute
of Mathematics, Seoul National University, Seoul 151-747, Korea}
\email{kiem@math.snu.ac.kr}

\author[Jun Li]{Jun Li}
\address{Shanghai Center for Mathematical Sciences, Fudan University, China; \hfil\newline 
\indent Department of Mathematics, Stanford University,
USA} \email{jli@math.stanford.edu}

\thanks{  HLC was partially supported by Hong Kong GRF grant 600711; YHK was partially supported by Korea NRF grant 2011-0027969; JL was partially supported by  
NSF grant DMS-1104553 and DMS-1159156.}

\begin{document}
\maketitle

\section{Introduction}\label{S:Intro}
Since its introduction in 1995 by Li-Tian \cite{LiTi} and Behrend-Fantechi \cite{BeFa}, the theory of virtual fundamental classes has played a key role in algebraic geometry, defining important invariants such as the Gromov-Witten invariant and the Donaldson-Thomas invariant. Quite a few methods for handling the virtual fundamental classes were discovered such as torus localization (\cite{GrPa}), degeneration (\cite{LiD}), virtual pullback (\cite{Mano1}) and cosection localization (\cite{KiemLi}). Often combining these methods turns out to be quite effective.
The purpose of this paper is to prove \begin{enumerate}
\item virtual pullback formula,
\item torus localization formula and
\item wall crossing formula
\end{enumerate} 
for cosection localized virtual cycles. Our results can be thought of as generalizations of the corresponding results for the ordinary virtual cycles because when the cosection is trivial, these formulas coincide with those for the ordinary virtual cycles. For (2), we remove a technical assumption in \cite{GrPa} on the existence of an equivariant global embedding into a smooth \DM stack.

\medskip
A \DM stack $X$ is equipped with the intrinsic normal cone $\fC_X$ which is \'etale locally $[C_{U/V}/T_V|_U]$ if $U\to X$ is \'etale and $U\hookrightarrow V$ is an embedding into a smooth variety where $C_{U/V}$ is the normal cone of $U$ in $V$.
A perfect obstruction theory (\cite{BeFa}) gives us a vector bundle stack $\fE$ together with an embedding $\fC_X\subset \fE$. 
The virtual fundamental class (\cite{BeFa, LiTi}) is then defined by applying the Gysin map to the intrinsic normal cone
$$[X]\virt=0^!_\fE[\fC_X].$$
When there is a torus action on $X$ with respect to which the perfect obstruction theory is equivariant, the virtual fundamental class is 
localized to the fixed locus $F=X^{\CC^*}$ under suitable assumptions ({\cite{GrPa}):
\beq\label{1e1}
[X]\virt=\imath_*\frac{[F]\virt}{e(N\virt)}\in A_\ast^{\CC^*} X\otimes_{\QQ[t]}\QQ[t,t\upmo].
\eeq
Here $\iota: F\to X$ is the inclusion and $t$ is the generator of the equivariant ring of $\CC^*$.

\medskip
The construction of virtual fundamental class can be relativized for morphisms $f:X\to Y$ to give the virtual pullback
$$f^!:A_*(Y)\to A_*(X)$$
when the intrinsic normal cone $\fC_{X/Y}$ is embedded into a vector bundle stack $\fE$ on $X$. 
When the perfect obstruction theories of $X$ and $Y$ are compatible with $\fE$, the virtual pullback gives us the formula (\cite{Mano1}) 
\beq\label{1e2}
f^![Y]\virt=[X]\virt.
\eeq

A wall crossing formula (\cite{KL3}) compares $[M_+]\virt$ with $[M_-]\virt$ when $M_+$ and $M_-$ are open \DM substacks of the quotient $[X/\CC^*]$ of a \DM stack $X$ which are simple $\CC^*$-wall crossings. 

The cosection localization says that when there is an open $U\subset X$ and a surjective 
$\sigma:\fE|_U\to \CC_U$, we can define a 
cosection localized virtual fundamental class $$[X]\virtloc\in A_*(X(\sigma))\quad \text{where }X(\sigma)=X-U$$
which satisfies usual expected properties such as deformation invariance and 
$$\imath_*[X]\virtloc=[X]\virt\in A_*(X)\quad \text{where } \imath:X(\sigma)\hookrightarrow X.$$
The construction of $[X]\virtloc$ in (\cite{KiemLi}) is obtained in two steps: 
\begin{itemize}
\item (cone reduction) the intrinsic normal cone $\fC_{X}$ has support contained in $\fE(\sigma)$ where $\fE(\sigma)=\fE|_{X(\sigma)}\cup \ker(\fE|_U\to  \CC_U)$;
\item (localized Gysin map) there is a cosection localized Gysin map $$0_{\fE,\loc}^!:A_*(\fE(\sigma))\lra A_*(X(\sigma))$$ compatible with the usual Gysin map.
\end{itemize}
Then the cosection localized virtual fundamental class is defined as  $$[X]\virtloc=0^!_{\fE,\loc}[\fC_X].$$
The cosection localized virtual fundamental class turned out to be quite useful (\cite{Buss1, CK, CLL, Clad, GS1, GS2, HLQ, JT, KL1, KL2, KoTh, LQ, MPT, PP, PT}). 
For further applications, it seems desirable to have 
cosection localized analogues for torus localization formula, virtual pullback and wall crossing formulas.
For instance, recently there arose a tremendous interest in the Landau-Ginzburg theory whose key invariants such as the Fan-Jarvis-Ruan-Witten invariants are defined algebro-geometrically by cosection localized virtual cycles (\cite{CLL, CKL}). The formulas proved in this paper will be useful in the theory of MSP fields developed in \cite{CLLL} tod
study the Gromov-Witten invariants and the Fan-Jarvis-Ruan-Witten invariants of quintic Calabi-Yau threefolds.

\medskip
In \S\ref{S:Man}, we prove the cosection localized virtual pullback formulas (cf. Theorems \ref{1thmain1} and \ref{1thmain2}). The proofs in \cite{Mano1} work with necessary modifications as long as the rational equivalences used in the proofs lie in the suitable substacks for localized Gysin maps. In \S\ref{S:Torus}, we prove the torus localization formula for the cosection localized virtual fundamental classes (cf. Theorem \ref{2thmain}). In this new proof, we do not require (1) the existence of an equivariant global embedding of $X$ into a smooth \DM stack and (2) the existence of a global resolution of the perfect obstruction theory as in \cite{GrPa}. The technical condition (1) is completely gone while (2) is significantly weakened to ($2'$) the existence of a global resolution of the virtual normal bundle $N\virt$ only on the fixed locus $F$. Finally, in \S\ref{S:Wall}, we prove a wall crossing formula for cosection localized virtual fundamental classes. We remark that in \cite{KL2}, a degeneration formula for cosection localized virtual fundamental class was proved and it was effectively used to prove Maulik-Pandharipande's formulas for Gromov-Witten invariants of spin surfaces. 

All schemes or \DM stacks in this paper are defined over the complex number field $\CC$.


\section{Virtual pullback for cosection localized virtual cycles}\label{S:Man}
In this section, we show that Manolache's virtual pullback formula (\cite{Mano1}) holds for cosection localized virtual cycles (cf. Theorems \ref{1thmain1} and \ref{1thmain2}). 

\def\cS{\mathcal{S} }
\subsection{Virtual pullback of cosection localized virtual cycle}\label{ss2.1}
Let $f:X\to Y$ be a morphism of \DM stacks. 
Let  $\phi_X:E_{X}\to \bbL_{X}$ and $\phi_Y:E_{Y }\to \bbL_{Y }$ be (relative) perfect obstruction theories that fit into a commutative diagram of distinguished triangles
\beq\label{16}\xymatrix{
f^*E_{Y }\ar[r]^\varphi\ar[d]_{f^*\phi_Y} & E_{X } \ar[r]\ar[d]_{\phi_X} & E_{X/Y}\ar[r]\ar[d]^{\phi_{X/Y}} & \\ 
f^*\bbL_{Y }\ar[r] & \bbL_{X }\ar[r] & \bbL_{X/Y}\ar[r] &
}\eeq
\begin{defi}\label{11}
We say $f:X\to Y$ is \emph{virtually smooth} if $E_{X/Y}$ is perfect of amplitude  $[-1,0]$.
\end{defi}

By \cite[\S3.2]{Mano1},
if $f$ is virtually smooth, then $\phi_{X/Y}$ is a perfect obstruction theory.
In the remainder of this section, we assume $f:X\to Y$ is virtually smooth. 
By \cite{BeFa}, the perfect obstruction theory $\phi_{X/Y}:E_{X/Y}\to \bbL_{X/Y}$ gives us an embedding of the intrinsic normal sheaf into the vector bundle stack
$$h^1/h^0(\bbL_{X/Y}^\vee)\hookrightarrow h^1/h^0(E_{X/Y}^\vee)=:\fE_{X/Y}.$$
Moreover the intrinsic normal cone of the morphism $f$ is naturally embedded  
into the intrinsic normal sheaf 
$$\fC_{X/Y}\hookrightarrow \fE_{X/Y}.
$$

Let ${Ob}_{X }=h^1(E_{X }^\vee)$ and $Ob_{Y }=h^1(E_{Y }^\vee)$ be the obstruction sheaves and let $\sigma_Y:Ob_Y\to \sO_Y$ be a cosection.
The morphism $f^*E_{Y }\to E_{X }$ induces a morphism $Ob_{X }=h^1(E_{X }^\vee)\to h^1(f^*E_Y^\vee)= f^*h^1(E_{Y }^\vee)=f^*Ob_{Y }$. Hence $\sigma_Y$ induces a cosection 
$$\sigma_X:Ob_{X }\lra f^*Ob_{Y }\lra f^*\sO_Y=\sO_X$$
of $Ob_{X }$. We call $\sigma_X$ the cosection induced from $\sigma_Y$. 

\begin{defi} We denote by $X(\sigma)=\sigma_X^{-1}(0)$ and $Y(\sigma)=\sigma_Y^{-1}(0)$ the vanishing loci of the cosections $\sigma_X$ and $\sigma_Y$ respectively.
\end{defi}

Here $\sigma_X^{-1}(0)$ is the subvariety defined by the image (ideal) of $\sigma_X$.

\begin{lemm}\label{13} If $f$ is virtually smooth, $X(\sigma)=f^{-1}(Y(\sigma))=Y(\sigma)\times_Y X$.\end{lemm}
\begin{proof}
From the distinguished triangle $E_{X/Y}^\vee\to E_{X }^\vee\to f^*E_{Y }^\vee\mapright{+1}$, we obtain an exact sequence
\[ \cdots\lra Ob_{X }\lra f^*Ob_{Y }\lra h^2(E_{X/Y}^\vee)=0\]
where the identity holds because $E_{X/Y}$ is perfect of amplitude  $[-1,0]$. Since $Ob_{X }\to f^*Ob_{Y }$ is surjective, 
\[ \sigma_X:Ob_{X }\twoheadrightarrow f^*Ob_{Y }\mapright{f^*\sigma_Y} \sO_X\]
is zero if and only if $f^*\sigma_Y$ is zero. 
This proves the lemma.
\end{proof}

By Lemma \ref{13}, we have a Cartesian square
\beq\label{17}\xymatrix{
X(\sigma)\ar[d]_{\imath'}\ar[r]^{f_\sigma} & Y(\sigma)\ar[d]^{\imath}\\
X\ar[r]_f & Y 
}\eeq
where the vertical arrows are inclusion maps.

\medskip
We recall Manolache's virtual pullback.
\begin{defi}\label{14} \cite{Mano1} 
Suppose we have an embedding of the intrinsic normal cone $\fC_{X/Y}$ of $f:X\to Y$ into a vector bundle stack $\fE_{X/Y}$. Consider a fiber product diagram
\[\xymatrix{
X'\ar[r]^{f'} \ar[d]_{p} & Y'\ar[d]^q\\
X\ar[r]^f& Y.
}\]
Then the \emph{virtual pullback} is defined as the composite
$$f^!:A_*(Y')\to A_*(\fC_{X'/Y'})\to A_*(\fC_{X/Y}\times_XX')\to A_*(p^*\fE_{X/Y})\to  A_{*+d}(X'),
$$
where the first arrow is $[B]\mapsto [\fC_{B\times_{Y'}X'/B}]$; the second arrow is via the inclusion $\fC_{X'/Y'}\to
\fC_{X/Y}\times_XX'$; the last arrow is the Gysin map $0^!_{\fE_{X/Y}}: A_*(p^*\fE_{X/Y})\to  A_{*+d}(X')$
for the vector bundle stack $\fE_{X/Y}$. Here $d$ is the rank of $E_{X/Y}$.
\end{defi}

If $X$ is not connected, we consider each connected component separately.

Letting $X'=X$ and $Y'=Y$, we get $f^!:A_*(Y)\to A_{*+d}(X)$. Letting $X'=X(\sigma)$ and $Y'=Y(\sigma)$, we obtain $f_\sigma^!:A_*(Y(\sigma))\to A_*(X(\sigma))$. By \cite[Theorem 2 (i)]{Mano1}, these fit into a commutative diagram
\beq\label{18}\xymatrix{
A_*(Y(\sigma))\ar[r]^{f_\sigma^!}\ar[d]_{\imath_*} & A_{*+d}(X(\sigma))\ar[d]^{\imath'_*}\\
A_*(Y)\ar[r]_{f^!} & A_{*+d}(X)
}\eeq

For the virtual pullback formula, we need the following analogue of \cite[Lemma 4.7]{Mano1}.
\def\ker{\mathrm{ker} }
\begin{lemm}\label{110}
Let $f:X\to Y$ be a morphism of \DM stacks and ${\cN}$ be a vector bundle stack on $X$ such that $\fC_{X/Y}\subset {\cN}$. Let $\fE$ be a vector bundle stack on $Y$ with the zero section $0_\fE:Y\to \fE$.
Let $U\subset Y$ be open and $\sigma:\fE|_U\to \CC_U$ be a surjective map of vector bundle stacks. Let $Y(\sigma)=Y-U$;
$X(\sigma)=X\times_YY(\sigma)$ and $f_\sigma: X(\sigma)\to Y(\sigma)$ the induced morphism. 
Let $\fE(\sigma)=\fE|_{Y(\sigma)}\cup \ker (\sigma)$.
Then $(\fC_{X/\fE})_{\mathrm{red}}\subset f^*\fE\oplus {\cN}$ where $\fC_{X/\fE}$ denotes the normal cone of the morphism 
$0_\cE\circ f: X\to Y\to \fE$.   
Moreover for each irreducible $B\subset \fE(\sigma)$,
\[ f_\sigma^!0^!_{\fE,\loc}[B]=0_{f^*\fE\oplus {\cN},\loc}^![\fC_{X\times_\fE B/B}] \quad\text{in } A_*(X(\sigma)) 
\]
where $0^!_{\fE,\loc}$ and $0_{f^*\fE\oplus {\cN},\loc}^!$ denote the localized Gysin maps with respect to the cosections 
$\sigma:\fE|_U\to \CC_U$ and $(f^*\sigma,0):f^*\fE\oplus{\cN}|_{f^{-1}(U)}\to \CC_{f^{-1}(U)}$ respectively.
\end{lemm}

\begin{proof}
The inclusion $\fC_{X/\fE}\subset f^*\fE\oplus {\cN}=f^*\fE\times_X{\cN}$ follows from the identity $\fC_{X/\fE}=f^*\fE\times_X\fC_{X/Y}$ proved in \cite[Example 2.37]{Mano1}. We prove the last statement. If $B\subset \fE|_{Y(\sigma)}$, the localized Gysin maps are the ordinary Gysin maps and hence the lemma follows from \cite[Lemma 4.7]{Mano1}. So we may suppose $B\nsubseteq \fE|_{Y(\sigma)}$. 
Further, by shrinking $Y$ if necessary, we can assume that $Y$ is integral and $B\to Y$ is dominant.

By \cite[\S2]{KiemLi}, we can choose a projective variety $Z$, a generic finite and proper morphism $\rho:Z\to Y$, a Cartier divisor $D$ on $Z$ such that $D$ is a linear combination of integral Cartier divisors; $D$ fits into the commutative diagram
\[\xymatrix{
Z\ar[r]^\rho & Y\\
D\ar@{^(->}[u]\ar[r]^{\rho_\sigma} & Y(\sigma)\ar@{^(->}[u]
}\]
and  $\rho^*\sigma$ extends to a surjective map $\tilde{\sigma}:\rho^*\fE\to \sO_Z(D)$, where by abuse of notation we think of
$\sO_Z(D)$ as the total space of the line bundle $\sO_Z(D)$.

Let $\tilde\fE=\ker(\tilde\sigma)$, which is a bundle stack, and choose an integral $\tilde B\subset \tilde\fE$ such that $\tilde\rho_*[\tilde B]=k[B]$ for some $k>0$, where $\tilde\rho:\tilde\fE\to\fE(\sigma)$ is the composition $\tilde\fE\subset \rho^*\fE\to \fE$ of the inclusion with the first projection $\rho^*\fE=\fE\times_YZ$. Then by the definition in \cite[\S2]{KiemLi}, 
$$0^!_{\fE,\loc}[B]=\frac1{k}{\rho_\sigma}_*(0^!_{\tilde\fE}[\tilde B]\cdot D),
$$
where $\cdot D$ denotes the refined intersection for the inclusion $D\subset Z$ defined in \cite[Chapter 6]{Fulton}.

We further simplify the situation as follows. Because $\ti\cE$ is a bundle stack, there are integral $Z_i\sub Z$
and rational $c_i$ so that $[\ti B]=\sum c_i[\ti\cE|_{Z_i}]\in A\lsta\ti\cE$. Because rational equivalence in $\ti\cE$ induces
a rational equivalence in $\cE(\sigma)$, to prove the theorem, we only need to consider the case where $\ti B=\ti\cE$.   
Thus the above identity becomes $0^!_{\fE,\loc}[B]=\frac1{k}{\rho_\sigma}_*[D]$.

Consider the Cartesian diagrams
\[\xymatrix{
D\ar[r]^{\rho_\sigma} & Y(\sigma) \ar[r]^{\sub} & Y\\ 
D'\ar[u]^{f'_\sigma}\ar[r]_{\rho'_\sigma} & X(\sigma)\ar[u]_{f_\sigma}\ar[r]_{\sub} & X\ar[u]_f.
}\]
Since virtual pullbacks commute with pushforwards 
(cf. \cite[Theorem 4.1 (i)]{Mano1}), we have
$$ f_\sigma^!0^!_{\fE,\loc}[B]=\frac1{k} f_\sigma^!{\rho_\sigma}_*[D]=\frac1{k}{\rho'_\sigma}_*{f'_\sigma}^![D] = 
\frac1{k} {\rho'_\sigma}_*0^!_{{\cN}|_{D'}}[\fC_{D'/D}].$$

Consider the commutative diagram
\[\xymatrix{
Z'\ar[r]^{f'}\ar[d]^{\rho'} & Z\ar[r]^0\ar[d]^{\rho} & \tilde\fE\ar[d]^{\tilde\rho} & \tilde B\ar@{=}[l]\\
X\ar[r]^f & Y\ar[r]^{0} & \fE(\sigma) & B\ar[l]
}\]
where $Z':=X\times_Y Z$. 
Let 
$$\tilde\rho^f:{f'}^*\tilde\fE\to {f}^*\fE(\sigma)
$$ 
denote the pullback of $\tilde\rho$. 
Since $\tilde\rho_*[\tilde \cE]=k[B]$, we have 
$$\tilde\rho^f_*[\fC_{X\times_\fE \tilde \cE/\tilde \cE}]=k[\fC_{X\times_\fE B/B}].$$
By the definition of the localized Gysin map \cite[\S2]{KiemLi}, we have
$$0_{f^*\fE\oplus {\cN},\loc}^![\fC_{X\times_\fE B/B}]=
\frac1{k}{\rho'_\sigma}_*\left( 0^!_{{f'}^*\tilde\fE\oplus {\rho'}^*{\cN}}[ \fC_{X\times_\fE \tilde \cE/\tilde \cE}] \cdot D'\right).
$$
Here by $\cdot D'$ we mean the intersection with $D\sub Z$ via the Cartesian square
$$\begin{CD}
D'@>>> Z'\\
@VVV @VVf'V\\
D @>>> Z
\end{CD}
$$
Since $X\times_\fE\tilde\fE=Z'$, by \cite[Example 2.37]{Mano1}
$$\fC_{X\times_\fE \tilde \fE/\tilde \fE}=\fC_{Z'/\tilde\fE}=\fC_{Z'/Z}\times_{Z'}{f'}^*\tilde\fE.
$$
We thus have
$$0_{f^*\fE\oplus {\cN},\loc}^![\fC_{X\times_\fE B/B}]=\frac1{k}{\rho'_\sigma}_*\bigl(0^!_{{\cN}}[\fC_{Z'/Z}]\cdot D'\bigr)
=\frac1{k}{\rho'_\sigma}_*0^!_{{\cN|_{D'}}}\left([\fC_{Z'/Z}]\cdot D'\right)
$$ 
by the commutativity of Gysin maps. 

Let $L=\sO_Z(D)$ and let $L'={f'}^*L$.  We now prove that
$$[\fC_{D'/D}]=
[\fC_{Z'/Z}]\cdot D'.
$$
Indeed, by Vistoli's rational equivalence \cite{Vistoli} and \cite[Example 2.37]{Mano1}, 
and using $\fC_{D/Z}\cong L|_{D}$, we have
$$[\fC_{Z'/Z}]\cdot D'=0^!_{L'}[\fC_{D'/\fC_{Z'/Z}}]=0^!_{L'}[\fC_{D'/Z}]
=0^!_{L'}[\fC_{D'/\fC_{D/Z}}]
$$
$$=0^!_{L'}[\fC_{D'/D}\times_{D'}L'|_{D'}]=[\fC_{D'/D}].
$$
Therefore, we have
$$0_{f^*\fE\oplus {\cN},\loc}^![\fC_{X\times_\fE B/B}]=\frac1{k}{\rho'_\sigma}_*0^!_{{\cN|_{D'}}}[\fC_{D'/D}].$$
This proves the desired equality $ f_\sigma^!0^!_{\fE,\loc}[B]=0_{f^*\fE\oplus {\cN},\loc}^![\fC_{X\times_\fE B/B}].$
\end{proof}

The following is a cosection localized analogue of \cite[Corollary 4.9]{Mano1}.
\begin{theo}\label{1thmain1}
Let $f:X\to Y$ be a virtually smooth morphism of \DM stacks, and let $\sigma:Ob_{Y }\to \sO_Y$ be a cosection. Then
$$f_\sigma^![Y]\virtloc =[X]\virtloc.$$
\end{theo}
\begin{proof} 
The proofs of Theorem 4 (functoriality) and Corollary 4 in \cite{Mano1} work with necessary modifications. 
The reader is invited to go through the proofs in \cite{Mano1} with the proof of deformation invariance \cite[Theorem 5.2]{KiemLi} for cosection localized virtual cycles in mind. 

With Lemma \ref{110} at hand, one will find that the only thing to be checked is the inclusion of the support
\beq\label{19}\mathrm{Supp}\, \fC_{X\times \PP^1/M^0_{Y}}\subset \mathrm{ker}\left[ h^1/h^0(c(u)^\vee)\lra q^*\sO_{\PP^1}(-1) \right]\eeq
where $M^0_Y$ is the deformation space from the reduced point $\{pt\}$ to the intrinsic normal cone $\fC_Y$ and $c(u)$ is the cone of the morphism
$$
u=(x_0\cdot \mathrm{id}, x_1\cdot \varphi):p^*f^*E_Y\otimes q^*\sO_{\PP^1}(-1)\lra p^*f^*E_Y\oplus p^*E_X\quad\text{in}\quad D(X\times\PP^1).
$$
Here $x_0,x_1$ are the homogeneous coordinates of $\PP^1$; $p,q$ are projections from $X\times \PP^1$ to $X$ and  $\PP^1$ respectively; $\varphi:f^*E_Y\to E_X$ is the morphism in \eqref{16}. The morphism $ h^1/h^0(c(u)^\vee)\lra q^*\sO_{\PP^1}(-1) $ comes from the commutative diagram
$$\xymatrix{
h^1/h^0(c(u)^\vee)\ar[r]\ar[d] & p^*f^*h^1/h^0(E_Y^\vee)\oplus p^*h^1/h^0(E_X^\vee) \ar[r]\ar[d] & p^*f^*h^1/h^0(E_Y^\vee)\otimes q^*\sO_{\PP^1}(1)\ar[d]\\
q^*\sO_{\PP^1}(-1)\ar[r] & q^*\sO_{\PP^1}\oplus q^*\sO_{\PP^1} \ar[r] & q^*\sO_{\PP^1}(1)
}$$
where the middle and right vertical arrows come from the cosection $\sigma$. 

By the double deformation space construction, 
$$\fC_{X\times \PP^1/M^0_{Y}}\times_{\PP^1}(\PP^1-\{(0:1)\})=\fC_X\times (\PP^1-\{(0:1)\}).$$ 
By the cone reduction (\cite[\S4]{KiemLi}), we have the inclusion of the support 
$$\mathrm{Supp}\, \fC_X \subset \mathrm{ker}\left[ h^1/h^0(E^\vee_X)\to \sO_X \right].$$
Hence \eqref{19} holds over $\PP^1-\{(0:1)\}$. 

Let $D=\fC_{X\times \PP^1/M^0_{Y}}\times_{\PP^1}(\PP^1-\{(1:0)\})$ be open in $\fC_{X\times \PP^1/M^0_{Y}}$ containing the fiber over $(0:1)$. By diagram chase, 
$$h^1/h^0(c(u)^\vee)|_{\{(0:1)\}}\cong f^*h^1/h^0(E_Y^\vee)\times_X h^1/h^0(E_{X/Y}^\vee)$$ 
and the homomorphism in \eqref{19} over the point $(0:1)$ is 
$$(f^*\sigma,0):f^*h^1/h^0(E_Y^\vee)\times_X h^1/h^0(E_{X/Y}^\vee)\lra \sO_X.
$$
Therefore the theorem follows if we show that irreducible components $A$ of 
$D$ lying over $X\times \{(0:1)\}$ have support contained in
\beq\label{1501251}
\mathrm{ker}\left( 
f^*h^1/h^0(E_Y^\vee)\mapright{f^*\sigma} \sO_X \right)\times_X h^1/h^0(E_{X/Y}^\vee).
\eeq

Let $A$ be an irreducible component of $D$ lying over $X\times \{(0:1)\}$ and let $a$ be a general closed point in $A$. We claim that $a$ is contained in \eqref{1501251}.
Since the problem is local, we may assume $X$, $Y$ are affine, equipped with embeddings $X\subset V$, $Y\subset W$ into smooth affine varieties that fit into a commutative diagram
\[\xymatrix{
X\ar@{^(->}[r]\ar[d]_f & V\ar[d]^g\ar[r]^(.4){\gamma\circ g,\eta} & \bbA^m\times \bbA^n\ar[d]^{pr_1}\\
Y\ar@{^(->}[r] & W \ar[r]^\gamma & \bbA^m
}\]
such that the morphism $g:V\to W$ is a smooth, $X=\zero(\gamma\circ g,\eta)$, $Y=\zero(\gamma)$ and 
$$E_Y=[\sO_Y^{\oplus m}\mapright{d\gamma} \Omega_W|_Y], \quad E_{X/Y}=[\sO_X^{\oplus n}\mapright{d\eta} \Omega_{V/W}|_X].$$ 
Since we have nothing to prove when $\sigma=0$ at general points of the irreducible component $A$, we may assume that $\sigma$ is surjective.  

To prove the claim, we recall the double deformation space construction for $D$ (cf. \cite{KKP}): Let $\Gamma$ be the graph of the morphism
\[
V\times (\bbA^1_t-\{0\})\times (\bbA^1_s-\{0\})\lra \bbA^m\times\bbA^n, \quad (v,t,s)\mapsto ((ts)^{-1}\gamma\circ g(v), t^{-1}\eta(v))
\]
and let $\bar\Gamma$ be the closure of $\Gamma$ in $V\times \bbA^1_t\times\bbA^1_s\times\bbA^m\times \bbA^n$.
Here $\bbA^1_t$ and $\bbA^1_s$ denote the affine line with local coordinates $t$ and $s$ respectively. Then we have
$$\bar\Gamma\times_{\bbA^1_t}\{0\}/(pr_V^*T_V|_{\bar\Gamma\times_{\bbA^1_t}\{0\}}) = D.$$

Now we can prove the claim. Since $D\subset (\bar\Gamma-\Gamma)/(pr_V^*T_V|_{\bar\Gamma-\Gamma})$, we may choose a smooth pointed curve $(\Delta,0)$ with local coordinate $\delta$ and a morphism $\rho:\Delta\to \bar\Gamma$ such that $\rho(\Delta-\{0\})\subset \Gamma$ and $\rho(0)$ represents $a\in A$. Let $t_\Delta:\Delta\mapright{\rho} \bar\Gamma\mapright{pr_t} \bbA^1_t$ and $s_\Delta:\Delta\mapright{\rho} \bar\Gamma\mapright{pr_s} \bbA^1_s$ denote the compositions of $\rho$ and the projections to $\bbA^1_t$ and $\bbA^1_s$ respectively. Let $\rho_V:\Delta\to V$ denote the composition of $\rho$ with the projection $pr_V:\bar\Gamma\to V$ and let $v_0=\rho_V(0)\in X$. Then
$\rho(0)=(v_0,0,0,v_1,v_2),$ 
\beq\label{00}
v_1=\lim_{\delta\to 0} (t_\Delta s_\Delta)^{-1}\cdot \gamma\circ g\circ \rho_V\in \bbA^m,  \quad 
v_2= \lim_{\delta\to 0} t_\Delta^{-1}\cdot \eta\circ \rho_V \in \bbA^n.
\eeq
Since $\sO_Y^{\oplus m}\twoheadrightarrow Ob_Y\mapright{\sigma} \sO_Y$ is surjective, by copying the proofs of Lemma 4.4 and Corollary 4.5 in \cite{KiemLi}, we find that 
$v_1$ represents a point in $$\mathrm{ker}\left( 
f^*h^1/h^0(E_Y^\vee)\mapright{f^*\sigma} \sO_X \right)$$
and $v_2$ a point in $h^1/h^0(E_{X/Y}^\vee)$.
Therefore $a$ lies in \eqref{1501251}. This proves the theorem.
\end{proof}

\medskip

 \def\sD{{\mathscr D}}
\def\tsi{\tilde{\sigma}}

\subsection{Cosection localized virtual pullback}\label{ss2.2}
In \S\ref{ss2.1}, we considered the virtual pullback of a cosection localized virtual fundamental class
when there is a cosection $\sigma:Ob_Y\to \sO_Y$ on $Y$ that induces a 
cosection $Ob_X\to f^*Ob_Y\mapright{\sigma} \sO_X$ on $X$. Actually there is 
another way to combine virtual pullback with cosection localization. 
Consider the case when there is a cosection
$\sigma:Ob_X\to \sO_X$ that induces a cosection $\tsi:Ob_{X/Y}\to Ob_X\to \sO_X$ of the relative obstruction sheaf. 
In this subsection, we 
define the cosection localized virtual pullback 
$$f_\sigma^!:A_*(Y)\lra A_*(X(\tsi))$$
for a virtually smooth morphism $f:X\to Y$ where $X(\tsi)=\tsi^{-1}(0)$ (cf. Definition \ref{1clvp})
and prove the cosection localized virtual pullback formula (cf. Theorem \ref{1thmain2}).

\medskip

We let
$f: X\to Y$ be a virtually smooth morphism between \DM stacks as before; we let $\sigma=\sigma_X: \Ob_X\to\sO_X$
be a cosection, and form the (composite)  
$$\tsi=\sigma_{X/Y}: \Ob_{X/Y}=h^1(E_{X/Y}\dual)\lra h^1(E_X\dual)\mapright{\sigma}\sO_X.
$$
Let $X(\sigma)=\sigma^{-1}(0)$ and $X(\tsi)=\tsi^{-1}(0)$. Then by definition, we have an inclusion
$$\jmath:X(\sigma)\hookrightarrow X(\tsi).$$
 
We let 
$\cK=h^1/h^0(E_{X/Y}\dual)$. Then $\tsi$ induces a morphism $\cK\to \sO_X$ which we also denote by $\tsi$ by abuse of notation. As before, we denote
$$\cK(\tsi)=\ker[\tsi:\cK \to\sO_X]\defeq  \cK|_{X(\tsi)}\cup \ker[
\tsi|_U:\cK|_{U}\to \sO_{U}],
$$
where $U=X-X(\tsi)$ is the open where $\tsi$ is surjective.


\begin{lemm}\label{deg-vanishing}
We have
$$\mathrm{Supp} \,\fC_{{X}/ {Y}} \sub \cK(\tsi).
$$
\end{lemm}

\begin{proof} 
We apply the functoriality of the $h^1/h^0$ construction to \eqref{16}
to obtain the commutative diagram
\beq\label{h1h0}
\begin{CD}
 \fC_{{X}/ {Y}}@>\subset>>h^1/h^0(\TT_{{X}/ {Y}})@>\subset>>h^1/h^0(E\dual_{{X}/ {Y}})=\cK @>\tsi>>\sO_X\\
@VVV @VVV@VVV @V=VV\\
 \fC_{{X}{}}@>\subset>>h^1/h^0(\TT_{{X}}) @>\subset>> h^1/h^0(E\dual_{{X}})=:\cE @>\sigma>> \sO_X\\
\end{CD}
\eeq
Like before, we have
$\mathrm{Supp} \,\fC_{{X}} \sub \ker[h^1/h^0(E_{X}\dual)\lra\sO_X]=:\cE(\sigma)$.
Since $\tsi$ is induced from $\sigma=\sigma_X$, the lemma follows.
\end{proof}

We now define the cosection localized virtual pullback. We let
$Y'\to Y$ be a morphism of stacks where $Y'$ has stratification by global quotients. Form the Cartesian product
\[\xymatrix{
X'\ar[r]^p \ar[d]_{f'} & X\ar[d]^f\\
Y'\ar[r]^q& Y.
}\] 
and let $\hat\sigma: p\sta\cK\to \sO_{X'}$
be the pullback of $\tsi:\cK\to \sO_X$. Then the vanishing locus of $\hat\sigma$ is $$X'(\ti\sigma):=X(\tsi)\times_X X'$$
and $$p^*\cK(\hat\sigma)=\ker[\hat\sigma:p^*\cK\to \sO_{X'}]=\cK(\tsi)\times_XX'.$$

Consider the composite
$$\iota: (\fC_{X'/Y'})_\redd \sub (\fC_{{X}/ {Y}}\times_{ {X}} {X'})_\redd \sub \cK(\tsi)\times_{ {X}} {X'}
= p^*\cK(\hat\sigma),
$$
where the first inclusion follows from the definition of $X'$, and the second inclusion follows from Lemma
\ref{deg-vanishing}. 
\begin{defi}\label{1clvp} The \emph{cosection localized virtual pullback} is defined by
$$f_\sigma^! : A\lsta Y'\mapright{\ep}A\lsta \fC_{X'/Y'}\mapright{\iota\lsta}
A\lsta(p\sta\cK(\hat\sigma))\mapright{0^!_{p^*\cK,\loc}} A_{\ast}(X'(\tsi)),
$$
where $\ep$ is defined on the level of cycles by  $\ep(\sum n_i[V_i]) =\sum n_i[\fC_{V_i\times_{Y'}X'/V_i}]$.  
 \end{defi}
Note that the way that \cite[Theorem 2.31]{Mano1} was applied to \cite[Construction 3.6]{Mano1} can also be
applied  here to conclude that $\ep$ descends to maps between Chow groups.

We have the following virtual pullback formula.
\begin{theo}\label{1thmain2} Let $f:X\to Y$ be a virtually smooth morphism, $\sigma:\Ob_X\to\sO_X$ be a cosection and 
$\tsi:\Ob_{X/Y}\to Ob_X\mapright{\sigma} \sO_X$ be the induced cosection. 
Let $\jmath: X(\sigma)\to X(\tsi)$ be the inclusion of zero loci of $\sigma$ and $\tsi$. Then we have
$$f_\sigma^![Y]\virt=\jmath\lsta [X]\virtloc\in A\lsta( X(\tsi)).
$$
\end{theo}
The proof is completely parallel to that of Theorem \ref{1thmain1}, so we only provide a sketch. 
We need the following analogue of Lemma \ref{110}.
\begin{lemm}\label{1100}
Let $f:X\to Y$ be a morphism of \DM stacks and ${\cK}$ be a vector bundle stack on $X$ such that $\fC_{X/Y}\subset {\cK}$. Let $\cF$ be a vector bundle stack on $Y$ with the zero section $0_\cF:Y\to \cF$.
Let $U\subset X$ be open and $\tsi:\cK|_U\to \CC_U$ be a surjective map of vector bundle stacks. Let $X(\tsi)=X-U$.
Then 
for each irreducible $B\subset \cF$,
\[ f_\sigma^!0^!_{\cF}[B]=0_{f^*\cF\oplus {\cK},\loc}^![\fC_{X\times_\cF B/B}] \quad\text{in } A_*(X(\tsi)) 
\]
where $0_{f^*\cF\oplus {\cK},\loc}^!$ denotes the localized Gysin map with respect to the cosection
$(0,\tsi):f^*\cF\oplus {\cK}|_{f^{-1}(U)}\to \CC_{f^{-1}(U)}$.
\end{lemm}
\begin{proof} 
We may assume that there is an irreducible $\tilde{B}\subset Y$ such that $B=\cF|_{\tilde{B}}=\cF\times_Y\tilde{B}$ and $0^!_{\cF}[B]=\tilde{B}$. The left side is  $$f_\sigma^!0^!_{\cF}[B]=f_\sigma^![\ti B]=0^!_{\cK,\loc} [\fC_{\ti B\times_YX/\ti B}]=0^!_{f^*\cF\oplus \cK,\loc} [f^*\cF\times_X \fC_{\ti B\times_YX/\ti B}].$$
Since $B=\cF\times_Y\tilde{B}$,  $f^*\cF\times_X \fC_{\ti B\times_YX/\ti B}=\fC_{B\times_\cF X/B}$. Hence
the lemma follows.
\end{proof}

Theorem \ref{1thmain2} follows from the following. 
\begin{prop}\label{vanishing2} 
Suppose $Y'=h^1/h^0(E_Y\dual)=\cF$ so that $p|_{X'(\tsi)}: X'(\tsi)\to X(\tsi)$ is a bundle stack
and that we have the Gysin map $0_{\cF}^! :A\lsta(X'(\tsi))\lra A_{\ast} (X(\tsi)).$
 Then we have
$$0_\cF^!f_\sigma^![\fC_{Y}]=\jmath\lsta [X]\virtloc\in A\lsta( X(\tsi)).
$$
\end{prop}
\begin{proof}[Proof of Theorem \ref{1thmain2}] By Lemma \ref{1100} and Proposition \ref{vanishing2},
$$f_\sigma^![Y]\virt=f_\sigma^!0_\cF^![\fC_Y]=
0_{f^*\cF\oplus {\cK},\loc}^![\fC_{X\times_\cF \fC_Y/\fC_Y}]
=0_\cF^!f_\sigma^![\fC_Y]=\jmath\lsta [X]\virtloc.$$
\end{proof}
\begin{proof}[Proof of Proposition \ref{vanishing2}]
The proof is almost identical to that of Theorem \ref{1thmain1}, so we only point out the difference. The construction of the double deformation space and the cone $c(u)$ is identical and we have a commutative diagram
\small $$\xymatrix{
h^1/h^0(c(u)^\vee)\ar[r]\ar[d] & p^*f^*h^1/h^0(E_Y^\vee)\oplus p^*h^1/h^0(E_X^\vee) \ar[r]\ar[d]^{(0,\sigma)} & p^*f^*h^1/h^0(E_Y^\vee)\otimes q^*\sO_{\PP^1}(1)\\
q^*\sO_{\PP^1}\ar[r]^= & q^*\sO_{\PP^1}  & 
}$$\normalsize
where the vertical arrows are defined by $\sigma$.
Again it suffices to show 
\beq\label{claim}
\mathrm{Supp}\, \fC_{X\times\Po/M_Y^\circ}\sub \ker[h^1/h^0(c(u)\dual)\to q\sta\sO_{\Po}].
\eeq
Now the proof continues exactly the same as the proof of Theorem \ref{1thmain1}.
By the cone reduction in \cite{KiemLi}, \eqref{claim} holds over the open $\PP^1-\{(0:1)\}$.
To prove \eqref{claim} over the point $(0:1)$, we consider a general point $a$ in any irreducible component $A$ of $D$ 
lying over $(0:1)$ and use the local construction of the double deformation space. 
After choosing a morphism $\rho$ from a smooth pointed curve $(\Delta,0)$ with $\rho(0)$ representing $a$, 
one finds that we only have to check that $v_2$ represents a point in 
$$\ker \left(\tsi:h^1/h^0(E^\vee_{X/Y})\to h^1/h^0(E^\vee_X)\mapright{\sigma} \sO_X\right).
$$ 
Again this follows from the arguments in the proofs of Lemma 4.4 and Corollary 4.5 in \cite{KiemLi}:
When $v_2\ne 0$, because $\lim_{\delta\to 0}
s_\Delta=0$, the image $v'_2$ of $v_2$ in $h^1/h^0( E_{X}\dual)$ under the tautological
$$h^1/h^0( E_{X/Y}\dual)\lra h^1/h^0( E_{X}\dual)
$$
lies in $\fC_X$. Using the cosection $\sigma$, we see that $v'_2\in  \ker[h^1/h^0( E_{X}\dual)\to \sO_X]$.
Because $\tilde\sigma$ is induced by $\sigma$, we obtain \eqref{claim}.
This proves the proposition.
\end{proof}

\begin{rema} 
When $f:X\to Y$ is a morphism over a smooth Artin stack 
$\cS$, sometimes it is more convenient to work with relative obstruction theories, say with $\bbL_X$ replaced by 
$\bbL_{X/\cS}$, $\bbL_Y$ by $\bbL_{Y/\cS}$ etc. 
It is straightforward to see that all the statements and proofs in this section hold in this case.

Another useful situation is when $Y$ is only assumed to be an Artin stack with $Y\to\cS$ assumed to be Deligne-Mumford.
Then Proposition 2.11 holds in this case with obstruction theories replaced by relative (to $\cS$) obstruction theories.
\end{rema}

\section{Torus localization for cosection localized virtual cycles}\label{S:Torus}
In this section, we prove the torus localization formula for cosection localized virtual cycles (Theorem \ref{2thmain}). We do not assume the existence of an equivariant global embedding or a global resolution of the perfect obstruction theory. When the cosection is trivial, our argument gives a new proof of the torus localization theorem in \cite{GrPa} without these assumptions.

Let $X$ be a Deligne-Mumford stack acted on by a torus $T=\CC^*$. Let $F$ denote the $T$-fixed locus, i.e. locally if $X=\mathrm{Spec}(A)$, then $F=\mathrm{Spec}(A/\langle A^{\text{mv}}\rangle)$ where $A^{\mathrm{mv}}$ denotes the ideal generated by $T$-eigenfunctions with nontrivial characters. 
Let $$\imath:F\lra X$$ denote the inclusion map.

Let $D([X/T])$ be the derived category of sheaves of $T$-equivariant $\sO_X$-modules on $X$. It is the same as the ordinary derived category of sheaves of $\sO_X$-modules except that all sheaves are $T$-equivariant and all morphisms or arrows are $T$-equivariant. The action of $T$ on $X$ gives the equivariant cotangent complex $\bbL_X\in D([X/T])$.

\begin{defi} A \emph{$T$-equivariant perfect obstruction theory} consists of an object $E\in D([X/T])$ and a morphism $$\phi:E \to \bbL_X$$ 
in $D([X/T])$ which is a perfect obstruction theory on $X$.
\end{defi}

If $A$ is a $T$-equivariant sheaf of $\sO_F$-modules on $F$, we let $A^{\mathrm{fix}}$ denote the sheaf of $T$-fixed submodules and $A^{\mathrm{mv}}$ denote the subsheaf generated by $T$-eigensections with nontrivial characters. 
Given $E\in D([X/T])$, $\bar E:=E|_F$ is a complex of $T$-equivariant sheaves on $F$,
thus we can decompose $\bar E=\bar E^{\mathrm{fix}}\oplus \bar E^{\mathrm{mv}}$ into the fixed and moving parts. 
A $T$-equivariant chain map $\psi:\bar E\to \tilde E$ to an $\tilde E\in D([F/T])$ preserves such decompositions to give us 
$\psi^{\mathrm{fix}}:\bar E^{\mathrm{fix}}\to \tilde E^{\mathrm{fix}}$ and 
$\psi^{\mathrm{mv}}:\bar E^{\mathrm{mv}}\to \tilde E^{\mathrm{mv}}$. If $\psi$ is a quasi-isomorphism, so are $\psi^{\mathrm{fix}}$ and $\psi^{\mathrm{mv}}$. Therefore a $T$-equivariant perfect obstruction theory $\phi:E\to \bbL_X$ induces morphisms in $D([F/T])$
$$\phi^{\mathrm{fix}}:E|_F^{\mathrm{fix}}\lra \bbL_X|_F^{\mathrm{fix}}\and  \phi^{\mathrm{mv}}:E|_F^{\mathrm{mv}}\lra \bbL_X|_F^{\mathrm{mv}}.$$ 
\begin{lemm} Let the notation be as above.
The composition $\phi_F:E|_F^{\mathrm{fix}}\to \bbL_X|_F^{\mathrm{fix}}\to \bbL_F$ of $\phi^{\mathrm{fix}}$ and the natural morphism $\bbL_X|_F^{\mathrm{fix}}\to \bbL_F$ is a perfect obstruction theory of $F$.
\end{lemm}
\begin{proof}
For any square zero extension $\Delta\to \bar\Delta$ of $k$-schemes with ideal sheaf $J$ and a morphism 
$g:\Delta\to F$, let $\omega(g)\in Ext^1(g^*\bbL_F, J)$ denote the composition 
$g^*\bbL_F\to \bbL_\Delta\to J[1]$ of  the natural morphisms $g^*\bbL_F\to \bbL_\Delta$ from $g$ and $\bbL_\Delta\to \bbL_{\Delta/\bar\Delta}\to \bbL^{\ge -1}_{\Delta/\bar\Delta}=J[1]$ from 
$\Delta\to \bar\Delta$. Let 
$$\phi_F^*\omega(g)\in \Ext^1(g^*E|_F^{\mathrm{fix}},J)
$$ 
be the image of $\omega(g)$ by the map $\Ext^1(g^*\bbL_F,J)\to \Ext^1(g^*E|_F^{\mathrm{fix}},J)$ induced from $\phi_F: E|_F^{\mathrm{fix}}\to \bbL_F$. Note that $\Ext^1(g^*E|_F^{\mathrm{fix}},J)$ is a $T$-module and 
$\phi_F^*\omega(g)$ is $T$-invariant, where $T$ acts on $\Delta$, $\bar\Delta$ and $J$ trivially. 

By \cite[Theorem 4.5]{BeFa}, it suffices to show the following claim: the obstruction assignment
$\phi_F^*(\omega(g))$ vanishes if and only if an extension $\bar g:\bar\Delta\to F$ of $g$ exists; 
and if $\phi_F^*(\omega(g))=0$, then the extensions form a torsor under $\Ext^0(g^*E|_F^{\mathrm{fix}},J)$.

Let $h:\Delta\to X$ be the composite of $g$ with the inclusion $F\sub X$.
Since $\phi:E\to \bbL_X$ is a perfect obstruction theory, $h$ extends to $\bar h:\bar\Delta\to X$ if
and only if $0=\phi_X^*\omega(h)\in \Ext^1(h^*E,J)$. Because $h$ factors through $F\sub X$ and
$J$ is an $\sO_{\Delta}$-module, 
$$\Ext^1(h^*E,J)=\Ext^1(g^*E|_F^{\mathrm{fix}},J)\oplus \Ext^1(g^*E|_F^{\mathrm{mv}},J),
$$
as $T$-module, and further $\phi_X^*\omega(h)$ is $T$-invariant. Since $\phi_X^*\omega(h)^T=
\phi_F^*\omega(g)$, we see that $\phi_F^*\omega(g)=0$ if and only if $h$ extends to $\bar h: \bar\Delta\to X$.
Because $T$ is reductive, a standard argument shows that we can find a $T$-invariant extension $\bar h$,
which necessarily factors through $F\sub X$. This proves that $\phi_F^*\omega(g)$ is an obstruction class to
extending $g$ to $\bar g:\bar \Delta\to F$. 

The part on the space of extensions $\bar g$ follows by the same argument.
\end{proof}

We let $E_F:=E|_F^{\mathrm{fix}}$ and $N\virt:=(E|_F^{\mathrm{mv}})^\vee$. 
Since $E$ is perfect, both the fixed part $E_F$ and the moving part $E|_F^{\mathrm{mv}}=(N\virt)^\vee$ of $E|_F$ are perfect. They fit into the following diagram of distinguished triangles:
\[\xymatrix{
E|_F\ar[r] \ar[d] & E_F\ar[r]\ar[d] & (N\virt)^\vee[1]\ar[r]\ar[d] &\\
\bbL_X|_F\ar[r] & \bbL_F\ar[r] & \bbL_{F/X}\ar[r] &
}\]
The morphism $E|_F\to E_F$ induces a homomorphism $$Ob_F=H^1(E_F^\vee)\lra H^1(E|_F^\vee)\cong H^1(E^\vee)|_F=Ob_X|_F.$$

Let $\sigma:Ob_X=H^1(E^\vee)\to \sO_X$ be a $T$-equivariant cosection.  Then $\sigma$ induces a $T$-invariant cosection
$$\sigma_F:Ob_F\lra Ob_X|_F\lra \sO_X|_F=\sO_F$$
and we have a cosection localized virtual cycle $[F]\virtloc$.

\def\fN{\mathfrak{N} }


\begin{defi} Suppose the virtual normal bundle $N\virt$ admits a global resolution $[N_0\to N_1]$ 
by locally free sheaves $N_0$ and $N_1$ over $F$.
We define the Euler class $e(N\virt)$ of $N\virt$ to be
$$e(N\virt)=e(N_0)/e(N_1) \in A^*(F)\otimes_\QQ\QQ[t,t^{-1}].
$$
\end{defi}

The goal of this section is to prove the following.
\begin{theo}\label{2thmain}
Let $X$ be a Deligne-Mumford stack acted on by $T$ and let $E\to \bbL_X$ be an equivariant 
perfect obstruction theory on $X$. Let $F$ be the $T$-fixed locus in $X$. Let $\sigma:Ob_X\to \sO_X$ be a 
$T$-equivariant cosection on $X$. Suppose there is a global resolution $N\virt\cong[N_0\to N_1]$, where 
$N_0$ and $N_1$ are locally free sheaves on $F$ (whose ranks may vary from component to component). 
Then we have 
$$ [X]\virt_\loc = \imath_* \frac{[F]\virtloc}{e(N\virt)}\in A\lsta^T X\otimes_{\QQ[t]}\QQ[t,t\upmo].
$$
Here the class $[F]\virtloc$ is defined with respect to the induced perfect obstruction theory $E_F$ and cosection $\sigma_F$.
\end{theo}

\begin{rema} 
In \cite{GrPa},  the localization formula in Theorem \ref{2thmain} was proved for the ordinary virtual fundamental class under the following assumptions:
\begin{enumerate}
\item $X$ admits a global equivariant embedding into a smooth $Y$;
\item the perfect obstruction theory $E$ admits an equivariant global locally free resolution.
\end{enumerate} 
Both conditions are nontrivial unless $X$ is a projective scheme. Recent development in moduli theory and enumerative geometry utilizes a plethora of moduli stacks for which (1) is often tedious to verify and hence it is 
desirable to give a proof without the assumption (1). Here we remove the first assumption entirely and weaken
the second assumption to 
\begin{enumerate}
\item[($2'$)] the virtual normal bundle $N\virt$ admits a global locally free resolution $[N_0\to N_1]$ on the fixed locus $F$,
\end{enumerate}
which is often easier to check.
When $\sigma=0$, Theorem \ref{2thmain} says that the torus localization in \cite{GrPa} works without the assumption $(1)$ and with a much weaker $(2')$.\end{rema}

By our assumption that there is a resolution $[N_0\to N_1]$ of $N\virt$,  we find that the normal sheaf $N_{F/X}$ is contained in $h^1/h^0(N\virt[-1])=\ker\{ N_0\to N_1\}$, thus contained in $N_0$.
Hence the normal cone $\fC_{F/X}$ is contained in $N_0$ as well. As in Definition \ref{14}, we define the virtual pullback $$\imath^!:A_*(X(\sigma))\to A_*(F(\sigma))$$ for the inclusion $\imath:F\to X$, by
$$[B]\longmapsto [\fC_{B\times_XF/B}]\longmapsto 0^!_{N_0} [\fC_{B\times_XF/B}].$$
The proof of Theorem \ref{2thmain} is attained through the following two lemmas. 
\begin{lemm}\label{le2} 
Let $X(\sigma)$ and $F(\sigma)$ denote the vanishing loci of $\sigma$ and $\sigma_F$ respectively. Then $F(\sigma)=X(\sigma)\cap F$ and $\imath^![X]\virtloc=[F]\virtloc\cap e(N_1)$.
\end{lemm}
\begin{proof} The first identity follows from Lemma \ref{13}. We prove the second identity.
By definitions, $[X]\virtloc=0^!_{\fE,\loc}[\fC_X]$ and $[F]\virtloc=0^!_{\fE_F,\loc}[\fC_F]$, where 
$\fE=h^1/h^0(E^\vee)$ and $\fE_F=h^1/h^0(E_F^\vee)$. By Lemma \ref{110}, we have 
$$\imath^!0^!_{\fE,\loc}[\fC_X]=0^!_{\fE|_F\oplus N_0,\loc}[\fC_F]$$
because $[\fC_{F\times_\fE\fC_X/\fC_X}]=[\fC_{F/\fC_X}]=[\fC_F]$ by Vistoli's rational equivalence \cite{Vistoli}. 
The proof of Theorem \ref{1thmain1} guarantees that the rational equivalence lives in the desired locus for the localized Gysin maps. 
Therefore 
$$\imath^![X]\virtloc=\imath^!0^!_{\fE,\loc}[\fC_X]=0^!_{\fE|_F\oplus N_0,\loc}[\fC_F]
=0^!_{\fE_F\oplus N_1,\loc}[\fC_F]=[F]\virtloc\cap e(N_1),
$$
because $\fE|_F=\fE_F\oplus [N_1/N_0]$.
\end{proof}

\begin{lemm}\label{le3}
$\imath^!\imath_*\alpha =\alpha \cap e(N_0)$ for $\alpha\in A_*(F(\sigma))\otimes_\QQ\QQ[t,t^{-1}]$.
\end{lemm}
\begin{proof}
If $B$ is a cycle in $F(\sigma)$, the normal cone of $B\cap F(\sigma)$ in $B$ is $B$ and $\imath^!\imath_*B=0^!_{N_0}B=B\cap e(N_0)$ by the definition of virtual pullback $\imath^!$. 
\end{proof}

Now we can prove Theorem \ref{2thmain}.

\begin{proof}[Proof of Theorem \ref{2thmain}]
By \cite[Theorem 6.3.5]{Kres}, 
$$\imath_*:A^T_*(F(\sigma))\otimes_{\QQ[t]}\QQ[t,t^{-1}] \lra A^T_*(X(\sigma))\otimes_{\QQ[t]}\QQ[t,t^{-1}] 
$$ 
is an isomorphism. Thus
$\imath_*\alpha=[X]\virtloc\quad$ for some 
$$\alpha\in A^T_*(F(\sigma))\otimes_{\QQ[t]}\QQ[t,t^{-1}]= A_*(F(\sigma))\otimes_\QQ \QQ[t,t^{-1}].
$$ By Lemmas \ref{le2} and \ref{le3}, 
$$[F]\virtloc \cap e(N_1)= \imath^![X]\virtloc =\imath^!\imath_*\alpha=\alpha\cap e(N_0).$$ 
Hence $$\alpha=\frac{[F]\virtloc}{e(N\virt)}\and [X]\virtloc=\imath_*\alpha=\imath_*\frac{[F]\virtloc}{e(N\virt)}$$ as desired.
\end{proof}

\begin{exam} Let $V=\CC^d$ be a vector space with $(z_1,\cdots, z_d)$ be its standard
coordinates. We let $T=\CC^*$ acts on $V$ via $(z_i)^\alpha=(\alpha z_i)$.
The global differentials $dz_i$ give a trivialization of $\Omega_V$, in the form $\Omega_V\cong
V\times V\sta$. Let $E=[T_V\mapright{0} \Omega_V]$ and 
$\sigma:\Omega_V=V\times V^*\to \sO_V$ be the tautological pairing. 
Then $\sigma^{-1}(0)=\{O\}\subset V$ is the (reduced) origin $O\in V$, and under the induced $T$-action
on $\Omega_V$, $\sigma$ is $T$-invariant. We observe
$$F=V^{\CC^*}=\{O\}, \quad E|_F=[V\mapright{0} V^*]=N\virt,\and E_F=[0\to 0].$$
Hence $e(N\virt)=(-1)^d$ and $[F]\virtloc$ is the zero cycle $[O]$ consisting of one simple point $O$. Then by
Theorem \ref{2thmain}, we have
$$[V]\virtloc=\frac{[F]\virtloc}{e(N\virt)}=(-1)^d [O]$$
as expected from \cite[Example 2.4]{KiemLi}.
\end{exam}

In this example, if instead we consider a cosection $\sigma': \Omega_V\to\sO_V$
via $dz_1\mapsto 1$ and $dz_{i>1}\mapsto 0$. Since $\sigma'$ is surjective, we obtain
$[V]\virtloc=0$. However, $\frac{[F]\virtloc}{e(N\virt)}=(-1)^d [O]$ as before. 
Hence  Theorem \ref{2thmain} does not apply.

\medskip

\section{Wall crossing formulas for cosection localized virtual cycles}\label{S:Wall}

In this section we provide a wall crossing formula for simple $\CC^*$-wall crossings. The construction and proof are rather standard (cf. \cite{KL3}).

Let $X$ be a Deligne-Mumford stack acted on by $T=\CC^*$. Let $\phi:E\to \bbL_X$ be a $T$-equivariant perfect obstruction theory, together with an equivariant cosection $\sigma:Ob_X=h^1(E^\vee)\to \sO_X$. 
Let \begin{enumerate}
\item $F$ be the $T$-fixed locus in $X$;
\item $X^s$ be the open substack of $x\in X$ so that the orbit $T\cdot x$ is 1-dimensional and closed in $X$;
\item $\Sigma^0_\pm=\{x\in X-(X^s\cup F)\,|\, \mathrm{lim}_{t\to 0}t^{\pm 1}\cdot x\in F\}$;
\item $\Sigma_\pm = \Sigma^0_\pm \cup F$;
\item $X_\pm = X-\Sigma_\mp \subset X$;
\item $M_\pm=[X_\pm/T]\subset M=[X/T]$ are separated Deligne-Mumford stacks.
\end{enumerate}
Recall from \S\ref{S:Torus} that we have the induced cosections $\sigma_F:Ob_F\to \sO_F$.
We then define the master space of the wall crossing $M_\pm$ to be
$$\fM=\left[ X\times \PP^1 - \Sigma_-\times\{0\} -\Sigma_+\times\{\infty\}/\CC^*\right]$$
where $\CC^*$ acts trivially on $X$ and by $\lambda\cdot(a:b)=(a:\lambda b)$ on $\PP^1$. The action of $T$ on $X$ induces an action of $T$ on $\fM$ whose fixed locus is
$$M_+\sqcup F\sqcup M_-$$
as is easy to check.
Since $\CC^*$ acts only on the component $\PP^1$, the pullback of any sheaf on $X$ by the projection $X\times \PP^1\to X$ is $\CC^*$-equivariant and hence descends to the free quotient $\fM$. By pulling back the perfect obstruction theory $\phi:E\to\bbL_X$ and descending, we obtain a morphism $\bar\phi:\bar E\to \bbL_\fM$. 
\begin{lemm}
The morphism $\bar\phi:\bar E\to \bbL_\fM$ is a $T$-equivariant perfect obstruction theory of $\fM$. Moreover the pullback of $\sigma$ descends to a $T$-equivariant cosection $\bar\sigma:Ob_\fM\to \sO_\fM$.
\end{lemm}
\begin{proof} This is straightforward and we omit the proof.\end{proof}

The cosection $\bar\sigma$ induces cosections on the fixed locus $M_\pm$ and $F$ in $\fM$.
We are ready to state the main result of this section.
\begin{theo}\label{3thmain}
Let the notation be as above. 
Suppose the virtual normal bundle $N\virt$ admits a resolution $[N_0\to N_1]$ by vector bundles on $F$.
Then we have
$$[M_+]\virtloc- [M_-]\virtloc= \mathrm{res}_{t=0} \frac{[F]\virtloc}{e(N\virt)}\quad\text{in } A^T_*(\fM)\otimes_{\QQ[t]}\QQ[t,t^{-1}].$$
\end{theo}
\begin{proof}
Applying Theorem \ref{2thmain} to the master space $\fM$, we find that 
$$[\fM]\virtloc = \frac{[M_+]\virtloc}{-t} +\frac{[M_-]\virtloc}{t} + \frac{[F]\virtloc}{e(N\virt)}$$
since the normal bundle of $M_+$ is trivial with weight $1$ while that of $M_-$ is trivial with weight $-1$ by construction. 
If we take $\mathrm{res}_{t=0}$, the left side vanishes because $[\fM]\virtloc\in A_*^T(\fM)$ has trivial principal part. Therefore the residue of the right side 
$$-[M_+]\virtloc +[M_-]\virtloc + \mathrm{res}_{t=0} \frac{[F]\virtloc}{e(N\virt)}$$
vanishes. This proves the theorem.
\end{proof}



\bibliographystyle{amsplain}

\begin{thebibliography}{99}



\bibitem{BeFa} K. Behrend and B. Fantechi. {\em The intrinsic normal cone.}  Invent. Math. \textbf{128} (1997), no. 1, 45–-88. 

\bibitem{Buss1} V. Bussi. {\em Generalized Donaldson-Thomas theory over fields K $\ neq $ C.} Preprint, arXiv:1403.2403.

\bibitem{Cao} Y. Cao and N. Leung. {\em Donaldson-Thomas theory for Calabi-Yau 4-folds.} Preprint,  arXiv:1407.7659.

\bibitem{CLLL} H.-L. Chang, J. Li, C.-C. Liu and W.-P. Liu, in preparation.

\bibitem{CK} H.-L. Chang and Y.-H. Kiem. \emph{Poincar\'e invariants are Seiberg-Witten invariants.} Geom. Topol. \textbf{17} (2013), no. 2, 1149–-1163. 

\bibitem{ChangLi} H.-L. Chang and J. Li. {\em Gromov-Witten invariants of stable maps with fields.}  Int. Math. Res. Not. IMRN 2012, no. 18, 4163–-4217.

\bibitem{CLL} H.-L. Chang, J. Li and W.-P. Li. {\em Witten's top Chern class via cosection localization}. Preprint,  arXiv:1303.7126.


\bibitem{CKL} J. Choi and Y.-H. Kiem. {\em Landau-Ginzburg/Calabi-Yau correspondence via quasi-maps, I}. In preparation.


\bibitem{Clad} E. Clader. {\em Landau-Ginzburg/Calabi-Yau correspondence for the complete intersections $X_{3,3}$ and $X_{2,2,2,2}$.} Preprint,  arXiv:1301.5530

\bibitem{EdGr} D. Edidin and W. Graham. {\em Equivariant intersection theory.} Invent. Math. \textbf{131} (1998), no. 3, 595-–634.

\bibitem{FJR} H. Fan, T. Jarvis and Y. Ruan. {\em The Witten equation, mirror symmetry, and quantum singularity theory.} Ann. of Math. (2) \textbf{178} (2013), no. 1, 1–106.

\bibitem{Fulton} W. Fulton. {\em Intersection theory.} Ergebnisse der Mathematik und ihrer Grenzgebiete. 3. Folge. 2. Springer-Verlag, Berlin, 1998.

\bibitem{GS1} A. Gholampour and A. Sheshmani. {\em Donaldson-Thomas Invariants of 2-Dimensional sheaves inside threefolds and modular forms.} Preprint,  arXiv:1309.0050.

\bibitem{GS2} A. Gholampour and A. Sheshmani. {\em Invariants of pure 2-dimensional sheaves inside threefolds and modular forms.} Preprint, arXiv:1305.1334.

\bibitem{GrPa} T. Graber and R. Pandharipande. {\em Localization of virtual cycles.}  Invent. Math. \textbf{135} (1999), no. 2, 487-–518. 

\bibitem{HLQ} J. Hu, W.-P. Li, Z. Qin. {\em The Gromov-Witten invariants of the Hilbert schemes of points on surfaces with $ p_g> 0$.} Preprint,  arXiv:1406.2472.




\bibitem{JT} Y. Jiang and R. Thomas.{\em Virtual signed Euler characteristics.} Preprint, arXiv:1408.2541.

\bibitem{KiemLi} Y.-H. Kiem and J. Li. {\em Localizing virtual cycles by cosections.} J. Amer. Math. Soc. \textbf{26} (2013), no. 4, 1025-–1050.

\bibitem{KL1} Y.-H. Kiem and J. Li. {\em  Low degree GW invariants of spin surfaces.} Pure Appl. Math. Q. \textbf{7} (2011), no. 4, Special Issue: In memory of Eckart Viehweg, 1449-–1475.
\bibitem{KL2} Y.-H. Kiem and J. Li. {\em Low degree GW invariants of surfaces II.} Sci. China Math. \textbf{54} (2011), no. 8, 1679–-1706.
\bibitem{KL3} Y.-H. Kiem and J. Li. {\em A wall crossing formula of Donaldson-Thomas invariants without Chern-Simons functional.} Asian J. Math. \textbf{17} (2013), no. 1, 63-–94.
\bibitem{KKP} B. Kim, A. Kresch and T. Pantev. {\em Functoriality in intersection theory and a
conjecture of Cox, Katz, and Lee}. J. Pure Appl. Algebra \textbf{179} (2003), no. 1-2, 127--136.
\bibitem{KoTh} M. Kool and R. {\em Thomas Reduced classes and curve counting on surfaces I: theory.} Preprint,  arXiv:1112.3069.
\bibitem{Kres} A. Kresch. {\em Cycle groups for Artin stacks.}  Invent. Math. \textbf{138} (1999), no. 3, 495–-536.

\bibitem{LiD} J. Li . {\em A degeneration formula of GW-invariants.} J. Differential Geom. \textbf{60} (2002), no. 2, 199–-293.
\bibitem{LiTi} J. Li and G. Tian. {\em Virtual moduli cycles and Gromov-Witten invariants of algebraic varieties.} J. Amer. Math. Soc. \textbf{11} (1998), no. 1, 119-–174.

\bibitem{LQ} W.-P. Li and Z. Qin. {\em The cohomological crepant resolution conjecture for the Hilbert-Chow morphisms.} Preprint,  arXiv:1201.3094.
\bibitem{Mano1} C. Manolachs. {\em Virtual pullbacks.} J. Algebraic Geom. \textbf{21} (2012), no. 2, 201–-245. 



\bibitem{MPT} D. Maulik, R. Pandharipande and R. Thomas. {\em Curves on K3 surfaces and modular forms.
With an appendix by A. Pixton.}
J. Topol. \textbf{3} (2010), no. 4, 937–-996.

\bibitem{PP} R. Pandharipande and A. Pixton. {\em Descendent theory for stable pairs on toric 3-folds.} J. Math. Soc. Japan\textbf{65} (2013), no. 4, 1337-–1372.

\bibitem{PT} R. Pandharipande and R Thomas. {\em The Katz-Klemm-Vafa conjecture for K3 surfaces.} Preprint, arXiv:1404.6698.



\bibitem{Vistoli} A. Vistoli. {\em Intersection theory on algebraic stacks and on their moduli spaces.} Invent. Math. \textbf{97} (1989), no. 3, 613-–670.

\end{thebibliography}

\end{document}